\documentclass[12pt]{amsart}

\usepackage{amscd}
\usepackage[sc,slantedGreek]{mathpazo}
\usepackage{graphics}
\usepackage{color}
\usepackage[all]{xy}
\usepackage{eucal}
\usepackage[T1]{fontenc}

\newcommand{\B}[1]{{\mathbb #1}}
\newcommand{\C}[1]{{\mathcal #1}}

\newtheorem{theorem}[subsection]{Theorem}%[section]
\newtheorem{corollary}[subsection]{Corollary}
\newtheorem{lemma}[subsection]{Lemma}
\newtheorem{proposition}[subsection]{Proposition}
\theoremstyle{definition}

\newtheorem{example}[subsection]{Example}
\theoremstyle{remark}
\newtheorem{remark}[subsection]{Remark}

\numberwithin{figure}{section}
\numberwithin{table}{section}
\newcommand{\al}{{\alpha}}

\newcommand{\Om}{{\Omega}}

\newcommand{\si}{{\sigma}}

\newcommand{\Mo}{(M,\omega )}

\newcommand{\cp}{\B C\B P}

\newcommand\BH{\operatorname{BHam}}

\newcommand\Ham{\operatorname{Ham}}

\begin{document}

\title[Hamiltonian characteristic classes]
{On the algebraic independence of Hamiltonian characteristic classes}
\author{\'Swiatos\l aw Gal}
\address{\'SG: University of Wroc\l aw}
\email{Swiatoslaw.Gal@math.uni.wroc.pl}
\author{Jarek K\k edra}
\address{JK: University of Aberdeen, University of Szczecin}
\email{kedra@abdn.ac.uk}
\author{Aleksy Tralle}
\address{AT: University of Warmia and Mazury in Olsztyn}
\email{tralle@matman.uwm.edu.pl}
%\thanks{thanks} 
%\keywords{keywords}
%\subjclass[2000]{Primary 53; Secondary 57}
%
%
\begin{abstract}
We prove that Hamiltonian characteristic classes
defined as fibre integrals of powers of the coupling class
are algebraically independent for generic coadjoint orbits.
\end{abstract}

\maketitle

\section{Introduction}
Let $\Mo$ be a closed symplectic manifold of dimension $2n$ and let
$$
\Mo \stackrel{i}\to E \stackrel{\pi}\to B
$$ 
be a Hamiltonian fibration over a simply connected base.  It means
that the structure group of the fibration is contained in the group
$\Ham\Mo$ of Hamiltonian diffeomorphisms of $\Mo$.  There exist a
cohomology class $\Om \in H^2(E)$ that is uniquely defined by the
following two conditions:
\begin{itemize}
\item[]
$i^*\Om = [\omega],$
\item[]
$\pi_!(\Om^{n+1})= 0.$
\end{itemize}
It is called the {\em coupling class}. The existence of the coupling 
class and its basic properties are discussed in Guillemin-Lerman-Sternberg
\cite{MR98d:58074} and Lalonde-McDuff \cite{MR1941438}.

Since the fibre integration is functorial the coupling class is
natural in the sense that the coupling class of the pull--back bundle
is the pull--back of the coupling class. One defines the following
characteristic classes of Hamiltonian fibration
$$
\mu_k(E) := \pi_!(\Om^{n+k}) \in H^{2k}(B).
$$
The fundamental question arises if these classes are nontrivial and
to what extent are they algebraically independent in the cohomology
ring $H^*(B\Ham\Mo)$ of the classifying space of the group of Hamiltonian
diffeomorphisms.

The first result about the algebraic independence was proved
by Reznikov for $\cp^n$ (Section 1.3 in \cite{MR2000f:53116})
and it states that the classes $\mu_k$ are algebraically
independent for $k=2,\ldots,n+1.$ Reznikov also suggested
that this result may be true for any coadjoint orbit
of a compact Lie group. The main result of the present note
states that Reznikov's claim is correct {\em generically}.

\begin{theorem}\label{T:main}
Let $G$ be a compact semisimple Lie group and let 
$$
\C K:=\{k\in \B N\,|\, \pi_{2k}(BG)\otimes \B Q \neq 0\}.
$$ There exists a nonempty Zariski open subset 
$A\subset \mathfrak g^{\vee}$ in the dual of the Lie algebra of 
$G$ such that for any $\xi \in A$ the coadjoint orbit $M_{\xi}$ of
$\xi$ satisfies the following.  The classes $\mu_k\in
H^{2k}(B\Ham(M_{\xi}))$ are algebraically independent for $k\in \C K.$
\end{theorem}

The proof of this theorem amounts to the calculation of the 
characteristic classes $\mu_k$ for the universal fibration
$M_{\xi}\to BG_{\xi}\to BG.$ Then we use the fact that the
cohomology ring $H^*(BG)$ is a polynomial ring generated by
elements in degrees in $\C K.$ 

\begin{corollary}\label{C:main}
Let $G$ be a compact simple Lie group different from $SO(4k)$ There
exists a nonempty Zariski open subset $A\subset \mathfrak g^{\vee}$ in
the dual of the Lie algebra of $G$ such that for any $\xi \in A$ the
coadjoint orbit $M_{\xi}$ of $\xi$ satisfies the following.  The
homomorphism $H^*(B\Ham(M_{\xi}))\to H^*(BG)$ induced by the action is
surjective and its image is generated by the classes $\mu_k$.
\end{corollary}

\begin{remark}
We exclude the orthogonal group $SO(4k)$ because there are two
generators in $H^{4k}(BSO(4k))$, the Euler class and the $k$-th
Pontryagin class. We shall show in Section \ref{SS:twistor} that the
Euler class is a nonzero multiple of the class $\mu_{2k}$ for a
bundle over $S^{4k}$ with fibre $SO(4k)/U(2k).$
\end{remark}

\begin{remark}
In order to prove the surjectivity of the homomorphism
$H^*(B\Ham\Mo)\to H^*(BG)$ one can use
characteristic classes defined by integrating products of the
equi\-variant Chern classes \cite{JK} or use different arguments
~\cite{MR2115670}.
\end{remark}

In Section \ref{S:examples}, we provide various examples of 
coadjoint orbits for which the algebraic independence 
does or does not hold. The simplest manifolds with 
the class $\mu_3$ trivial in $H^*(BG)$ are the complex grassmannian
$\operatorname{G}(2,4)$ of planes in $\B C^4$ and the flag
manifold $SU(3)/T$ for a certain invariant symplectic form.

\subsection{Conventions}
Throughout the paper $H^*(X)$ denotes the cohomology of $X$ with
real coefficients.

When we say that a {\em statement} holds for a {\em generic}
element $x\in V$ of an algebraic variety we mean that there
exists a Zariski open subset $S\subset V$ such that for every element
$x\in S$ the {\em statement} holds.

\subsection{Acknowledgements}
\'S.G. is partially supported by the {\sc mnisw} grant 
{\sc n n201 541738}.  J.K. would like to thank Ran Levi for useful
conversations which led to the proof of Lemma \ref{L:flag}.
A.T. would like to thank IHES and the Max-Planck-Institut for
hospitality during the work on this paper.

%%%%%%%%%%%%%%%%%%%%%%%%%%%%%%%%%%%%%%%%%%%%%%%%
\section{Proof of the main result}\label{S:proof}
%%%%%%%%%%%%%%%%%%%%%%%%%%%%%%%%%%%%%%%%%%%%%%%

\subsection{Flag manifolds}\label{SS:flag}
If a cohomology class evaluates nontrivially on a sphere then it is
not the sum of the products of classes of lower degree.  This simple
observation allows to prove the algebraic independence of the
characteristic classes for flag manifolds.

\begin{lemma}\label{L:flag}
Let $G$ be a compact and connected semisimple group.  Let $\omega $ be
a generic homogeneous symplectic form on the flag manifold $G/T$. Then
for a rationally nontrivial homotopy class $f\colon S^{2k} \to BG$
the induced Hamiltonian bundle has a nontrivial class $\mu_k$.
\end{lemma}

\begin{proof}
Consider the following pull-back diagram
$$
\xymatrix
{
G/T \ar[r]^= \ar[d]              & G/T \ar[d] \\
E   \ar[r]^{\hat f} \ar[d]^{\pi} & BT \ar[d]^{p} \\
S^{2k} \ar[r]^f     & BG
}
$$ 
Let $\sigma \in H^{2k}(S^{2k})$ be a generator and let 
$\sigma =f^*(\al)$ for some $\al\in H^{2k}(BG)$. Notice that this
implies that $k>1$.  Since the cohomology of $BT$ is generated by
classes of degree two the pull back 
$$
\pi^*(\si) = \pi^*(f^*(\al))=\hat f^*(p^*(\al))
$$ 
is a sum of products of classes of degree two.  This implies that the
cohomology ring $H^*(E)$ is generated by degree two classes.
Moreover, the inclusion of the fibre induces an isomorphism
$H^2(E)\cong H^2(G/T).$

We claim that there exists a Zariski open subset
of $H^2(E)$ consisting of classes whose highest power
is nonzero. Since taking the highest, i.e. $(n+k)$-th
power defined an algebraic map $H^2(E)\to H^{2(n+k)}(E)=\B R$
it is enough to find just one class with nontrivial highest
power. Here $\dim G/T = 2n$. Observe that the symmetric map 
$$
H^2(E)^{\otimes(n+k)}\ni a_1\otimes \cdots \otimes a_{n+k}
\mapsto a_1\cdot\ldots\cdot a_{k+n}\in H^{n+k}(E)
$$ 
is nontrivial as $H^*(E)$ is
generated in dimension 2 and $E$ is closed and oriented. 
Since any multilinear symmetric map is determined
(via the polarisation formula) by a polynomial map, the map
$$
H^2(E)\ni a\mapsto a^{n+k}\in H^{n+k}(E)
$$ 
is nontrivial.

Let $\omega\in \Omega^2(G/T)$ be a generic $G$-invariant symplectic
form. Then the associated coupling class $\Omega \in H^2(E)$ 
is such that $\Omega^{n+k}\neq 0$ and hence
$$
\left <\mu_k(E),\left[S^{2k}\right]\right> = 
\left <\pi_!(\Omega^{n+k}),\left[S^{2k}\right]\right> = 
\left <\Omega^{n+k},[E]\right> \neq 0. 
$$

\end{proof}

\begin{corollary}\label{C:flag}
Let $\C K=\{k\in \B N\,|\, \pi_{2k}(BG)\otimes \B Q \neq 0\}$.
For a generic homogeneous symplectic form on a flag manifold
$G/T$ the classes $\mu_k\in H^{2k}(\BH(G/T))$ 
are algebraically independent for $k\in \C K$. Moreover, these
classes cannot be generated by classes of smaller degrees.
\end{corollary}

\begin{proof}
Let $\xi \in \mathfrak g ^{\vee}$ be an element such that its
coadjoint orbit $M_{\xi}$ is diffeomorphic to a flag manifold $G/T.$
It follows from Lemma \ref{L:flag} that
for each $k\in \C K$ the class $\mu_k(E)\in H^*(S^{2k})$ is nonzero for
$\xi $ in an open and dense subset of $\mathfrak g^{\vee}$.  Taking
the intersection of these subsets for all $k\in \C K$ we obtain an
open and dense subset for which the classes are nontrivial for
Hamiltonian fibrations over spheres.
Since these classes don't vanish on spheres they cannot be generated
by classes of smaller degrees. The algebraic independence follows from
the fact that $H^*(BG)$ is a free polynomial algebra.
\end{proof}

\subsection{The general case}\label{SS:general}
The cohomology ring $H^*(BG)$ is isomorphic to the ring of the
invariants $S(\mathfrak g^{\vee})^G$ of the Lie algebra of $G$ and
to the ring of polynomials $S(\mathfrak t^{\vee})^W$ on the Lie
algebra of the maximal torus $T\subset G$ invariant under the
action of the Weyl group. Let $M_{\xi}$ be the coadjoint orbit
of $\xi\in \mathfrak g^{\vee}$. Then every class $\mu_k$ for $M_{\xi}$ 
uniquely defines an invariant polynomial
$p_{\xi,k}\in S(\mathfrak g)^G\cong S(\mathfrak t)^W$.

The function $H:M_\xi \to \mathbb R$ given by
$$
H(\operatorname{Ad}_g^{\vee}(\xi))=
\left<X,\operatorname{Ad}^{\vee}_g(\xi)\right>
$$ 
is the Hamiltonian function for the action on the coadjoint orbit
$M_\xi$ generated by $X\in \mathfrak g.$ The following lemma
follows from K\k edra-McDuff \cite[Lemma 3.6 and Lemma
3.9]{MR2115670}.

\begin{lemma}\label{L:polynomial}
Let $X\in \mathfrak g$.
The following formula holds true.
$$
p_{\xi,k}(X):= (-1)^k{{n+k} \choose k}\cdot 
\int_{G}\left<X,\operatorname{Ad}^{\vee}_g(\xi)\right>^k \operatorname{vol}_G
$$
\qed
\end{lemma}

\begin{corollary}\label{C:even}
Let $\Mo$ be a nontrivial coadjoint orbit of a semisimple Lie group
$G.$ Then the class $\mu_{2k}\in H^{4k}(B\Ham\Mo)$ is nontrivial for
every positive integer $k\in \B N.$ \qed
\end{corollary}

\subsection{Proof of Theorem \ref{T:main}}
It follows from Corollary \ref{C:flag} that the relevant classes $\mu_k$
are algebraically independent in $H^*(BG)$ for a generic $\xi$ defining a
flag manifold. Since the algebraic independence is an open condition it follows
that the polynomials $p_{\xi,k}$ are algebraically independent for a
generic $\xi \in \mathfrak g^{\vee}$ and $k\in \C K.$
\qed

\begin{remark}
Observe that $p_{\xi,k}(X)$ defines a bi-invariant polynomial
of degree $k$ on the tensor product $\mathfrak g\otimes \mathfrak g^{\vee}$.
The algebra of $G$-bi-invariant polynomials on $\mathfrak g\otimes \mathfrak g^{\vee}$
is isomorphic to the algebra of polynomials on
$\mathfrak t\otimes \mathfrak t^{\vee}$ bi-invariant for the Weyl group of $G$.
\end{remark}

There are two interesting Zariski open subsets of the dual Lie algebra
$\mathfrak t^{\vee}$ of the maximal torus $T\subset G$. One consists of those
elements giving flag manifolds as coadjoint orbits. And this is just
the union of the interiors of the Weyl chambers. The other consists
of elements such that the characteristic classes associated with their
coadjoint orbits are algebraically independent.  In the next section we
discuss examples showing that this two sets do not contain each
other. In particular, the algebraic independence holds not only
for flag manifolds.

%%%%%%%%%%%%%%%%%%%%%%%%%%%%%%%%%%%%%%%%%%%%%%
\section{Examples}\label{S:examples}
%%%%%%%%%%%%%%%%%%%%%%%%%%%%%%%%%%%%%%%%%%%%%

\subsection{Coadjoint orbits of $\operatorname{SU}(n)$}
\label{SS:su}
As we mentioned in the introduction Reznikov proved
that the classes $\mu_k$ are algebraically independent
in $H^*(B\Ham(\cp^{n-1}))$ for $k=2,\ldots,n$ 
(see K\k edra--McDuff \cite[Proposition 1.7]{MR2115670} 
for an alternative proof). Hence it
follows that these classes are also algebraically
independent for any coadjoint orbit of $SU(n)$ which
is close to $\cp^{n-1}$.

\subsection{The failure of the algebraic independence}
\label{SS:fail}

Checking directly whether a class $\mu_k$ is a polynomial in lower
degree classes seems to be complicated in general.  However, one can
make an interesting claim if $\dim H^{2k}(BG)=1.$ The latter condition
is always true for a simple compact Lie group $G$ and $k=2$, and for a
single odd number $k$ if $G$ has one of the groups $SU(n)$,
$SO(4k+2)$ or $E_6$ as a factor.  We shall discuss concrete examples
in the next sections.

\begin{proposition}\label{P:fail}
Let $G$ be a compact Lie group and let $m\in \B N$ be 
a number for which $\pi_{2m}(BG)\otimes \B Q = H^{2m}(BG;\B Q)=\B Q.$ 
Let $u \in S(\mathfrak g^{\vee})^G$ be a nontrivial invariant polynomial
of degree $m$.
The class $\mu_{m}\in H^{2m}(BG)$ is trivial for
the coadjoint orbit $M_\xi$ if and only if
$u(\xi)= 0.$ 
\end{proposition}

\begin{proof}
It follows from the hypothesis and the isomorphism
$H^*(BG) = S(\mathfrak g^{\vee})^G$ that the polynomial
$u$ is unique up to a constant. 
Since the polynomial $p_{\xi,m}(X)$ is bi-invariant on the tensor
product $\mathfrak {g}\otimes \mathfrak {g}^{\vee}$ there exists a
degree $m$ invariant polynomial $v$ on $\mathfrak {g}$ such that
$$
p_{\xi,m}(X) = u(\xi)\cdot v(X). 
$$
Because $p_{\xi,m}(-)$ is nontrivial for a generic $\xi$ we get that
the polynomial $v$ is nonzero.
Hence $p_{\xi,m}(X)$ is trivial if and only if $u(\xi)=~0.$
\end{proof}

\begin{corollary}\label{C:fail}
If $m$ in the above proposition is odd then there exists
a coadjoint orbit $M_{\xi}$ for which the class $\mu_{2m}$
is trivial in $H^{2m}(BG).$ \qed
\end{corollary}

\subsection{Coadjoint orbits of $\operatorname{SU}(n)$ again}
\label{SS:su_again}

\begin{proposition}\label{P:grassmannian}
The $\mu_3$ class in $H^6(BSU(n))$ is trivial for the
adjoint orbit of the diagonal matrix
$\operatorname{diag}[X_1,\ldots,X_n]\in \mathfrak{su}(n)$ if and only if
$\sum X_i^3=0$. In particular, the class
$\mu_3$ is trivial for 
the grassmannian $\operatorname{G}(m,2m)$ of $m$-planes
in $\B C^{2m}$ and certain flag manifolds.
\end{proposition}

\begin{proof}
The algebra $S(\mathfrak{su}(n))^{SU(n)}$ of invariant polynomials 
is generated by the polynomials of the form
$$
\mathfrak{su}(n)\ni X \mapsto \sum X_i^k,
$$
where $X_i\in \B C$ are the eigenvalues of the matrix $X$.
Hence any invariant polynomial of degree three is up to a
constant equal to~$\sum X_i^3.$

It follows from Proposition \ref{P:fail} that the orbit of a vector
for which the above polynomial is trivial has vanishing
class $\mu_3$. This include grassmannians $\operatorname{G}(m,2m)$
since it is (up to a scalar) the adjoint orbit of the diagonal matrix
$X=\operatorname{diag}[i,\ldots,i,-i,\ldots,-i].$ Taking a generic
zero of the above polynomial we obtain flag manifolds for which
the class $\mu_3$ is trivial.
\end{proof}

\subsection{Coadjoint orbits of $\operatorname{SO}(2n)$}
\label{SS:twistor}

Let $SO(2n)\to SO(2n+1) \to S^{2n}$ be the bundle of the
orthonormal frames with respect to the standard round 
metric. The associated bundle
$$
SO(2n)/U(n) \to E \to S^{2n}
$$
admits a fibrewise symplectic form due to Reznikov \cite{MR1225431}
(see also K\k edra-Tralle-Woike \cite{KTW} for a more general statement).
This symplectic form restricts to the symplectic form on the fibres
and it represents the coupling class. Since it is nondegenerate
the fibre integral of its top power is nonzero in $H^{2n}(S^{2n}).$
Thus the class $\mu_{n}(E)$ is a nonzero multiple of the
Euler class of the tangent bundle of $S^{2n}.$

It follows that the class $\mu_{n}$ is also nontrivial for an orbit
in a neighbourhood of $SO(2n)/U(n).$ In particular, for every 
subgroup $H=\prod U(n_i),$ with $\sum n_i=n,$ there is an 
$SO(2n)$--invariant form on $SO(2n)/H$ such that the class 
$\mu_n$ is nonzero.

\begin{example}\label{E:so(8)/u(4)}
Let $M=SO(8)/U(4).$ The above argument show that the class 
$\mu_4\in H^8(B\Ham(M))$ is indecomposable. Moreover, the class
$\mu_2$ is nontrivial, due to Corollary \ref{C:even}. Hence these
classes are algebraically independent and their pull backs generate
$H^*(BSO(8)).$
\end{example}

\subsection{Coadjoint orbits of $\operatorname{SO}(4n+2)$}
\label{SS:so(4n+2)}

Let $\mathfrak t \subset \mathfrak {so}(4n+2)$ be
the Lie algebra of the maximal torus given by the skew
symmetric matrices with $2\times 2$-blocks on the diagonal.
Hence an element $\xi \in \mathfrak t$ may be represented by
an $(2n+1)$-tuple $[t_1,\ldots,t_{2n+1}]$ of real numbers.

We have $\dim \pi_{4n+2}(BSO(4n+2))\otimes \B Q=1$ and the
corresponding characteristic class is defined by the Pfaffian.  When
restricted to the Lie algebra $\mathfrak t$, the Pfaffian is, up to a
constant, equal to the product of coordinates.  This proves the
following.

\begin{proposition}
Let 
$\xi=[t_1,t_2,\ldots,t_{2n+1}]\in \mathfrak t \subset \mathfrak {so}(4n+2).$
The class $\mu_{2n+1}\in H^{4n+2}(BSO(4n+2))$ is trivial for the
coadjoint orbit $M_{\xi}$ if and only if $\prod t_i=0.$
\qed
\end{proposition}

%%%%%%%%%%%%%%%%%%%%%%%%%%%%%%%%%%%%%%%%%%%%%%%%%%%%%%%%%%%%%
\section{An application to lattices in semisimple groups}
\label{S:lattices}
%%%%%%%%%%%%%%%%%%%%%%%%%%%%%%%%%%%%%%%%%%%%%%%%%%%%%%%%%%%%%

Let $G$ be a semisimple Lie group, $K\subset G$ a maximal compact
subgroup and $\Gamma \subset G$ an irreducible cocompact 
lattice trivially intersecting $K$. 
Let $M\subset G^c$ be a maximal compact subgroup of the
complexification of $G$. Let $H\subset K$ be the isotropy
subgroup of $\xi \in \mathfrak k^{\vee}.$ We have a Hamiltonian
bundle
$$
K/H \to \Gamma\backslash G/H \to \Gamma \backslash G/K = B\Gamma
$$
classified by a
map $\Gamma \backslash G/K \to BK.$ 
It was observed by Okun \cite{MR1875614} that this classifying map
lifts to a map $\Gamma \backslash G/K \to M/K$ after passing
to a sublattice of finite index if necessary. Let us call this
lift the Okun map. The homomorphism 
$H^*(M/K)\to H^*(\Gamma)$ induced by the Okun map is called
the Matsushima homomorphism and it is known to be injective.
It is also surjective in degrees smaller than  the rank of $G$.
This result was proved by Matsushima in \cite{MR0141138}.

We obtain the following diagram of Hamiltonian fibrations
$$
\xymatrix
{
K/H \ar[r] \ar[d] & K/H \ar[r] \ar[d] & K/H \ar[d]\\
\Gamma\backslash G/H \ar[r] \ar[d] & M/H \ar[r] \ar[d] & BH \ar[d]\\
B\Gamma \ar[r]^b & M/K \ar[r]^f & BK
}
$$

\begin{theorem}
Let $\xi \in \mathfrak k^{\vee}$ be a generic element such that its
isotropy subgroup $H\subset K$ is a maximal torus. If
$\pi_{2k}(M/K)\otimes \mathbb Q \to \pi_{2k}(BK)\otimes \mathbb Q$
is nontrivial then the characteristic class
$$
\mu_k(\Gamma\backslash G/H) \in H^{2k}(B\Gamma)
$$
is nonzero.
\end{theorem}

\begin{proof}
Let $\sigma \in \pi_{2k}(M/K)\otimes \mathbb Q$ be an element such
that $f_*(\sigma)\neq 0.$ According to Lemma \ref{L:flag} we have that
the class $\mu_k(BH)\in H^{2k}(BK)$ evaluates non trivially on
$f_*(\sigma)$ hence the corresponding characteristic class
$\mu_k(M/H)$ class is nonzero in $H^{2k}(M/K)$. Since the Matsushima
homomorphism is injective we obtain the statement.
\end{proof}

\begin{theorem}
Let $K$ be a simple compact group different from $SO(4n).$
Suppose it is of maximal rank in $G.$
Let $\xi \in \mathfrak k^{\vee}$ be a generic element. Denote its
isotropy subgroup by $H\subset K$.
Then the image of the Matsushima
homomorphism $H^*(M/K)\to H^*(B\Gamma)$ is generated
by the characteristic classes $\mu_k(\Gamma \backslash G/H)$.
\end{theorem}

\begin{proof}
It follows from Corollary \ref{C:main} that the classes $\mu_k$
generate the cohomology ring $H^*(BK)$.  Since $K$ and $M$ have equal
rank the homomorphism $H^*(BK)\to H^*(M/K)$ is surjective hence the
ring $H^*(M/K)$ is also generated by the classes  $\mu_k$ and so is its
image in $H^*(\Gamma)$.
\end{proof}

%\nocite{*}
\bibliography{../../bib/bibliography}

\def\polhk#1{\setbox0=\hbox{#1}{\ooalign{\hidewidth
  \lower1.5ex\hbox{`}\hidewidth\crcr\unhbox0}}}
  \def\polhk#1{\setbox0=\hbox{#1}{\ooalign{\hidewidth
  \lower1.5ex\hbox{`}\hidewidth\crcr\unhbox0}}}
\begin{thebibliography}{1}

\bibitem{MR98d:58074}
Victor Guillemin, Eugene Lerman, and Shlomo Sternberg.
\newblock {\em Symplectic fibrations and multiplicity diagrams}.
\newblock Cambridge University Press, Cambridge, 1996.

\bibitem{JK}
Tadeusz Januszkiewicz and Jarek K{\k e}dra.
\newblock Characteristic classes of smooth fibrations.
\newblock {\em math.SG/0209288}.

\bibitem{MR2115670}
Jarek K{\k{e}}dra and Dusa McDuff.
\newblock Homotopy properties of {H}amiltonian group actions.
\newblock {\em Geom. Topol.}, 9:121--162 (electronic), 2005.

\bibitem{KTW}
Jarek K{\k e}dra, Aleksy Tralle, and Artur Woike.
\newblock On nondegenerate coupling forms.
\newblock {\em {\tt http://arxiv.org/abs/1004.3699}}, 2010.

\bibitem{MR1941438}
Fran{\c{c}}ois Lalonde and Dusa McDuff.
\newblock Symplectic structures on fiber bundles.
\newblock {\em Topology}, 42(2):309--347, 2003.

\bibitem{MR0141138}
Yoz{\^o} Matsushima.
\newblock On {B}etti numbers of compact, locally sysmmetric {R}iemannian
  manifolds.
\newblock {\em Osaka Math. J.}, 14:1--20, 1962.

\bibitem{MR1875614}
Boris Okun.
\newblock Nonzero degree tangential maps between dual symmetric spaces.
\newblock {\em Algebr. Geom. Topol.}, 1:709--718 (electronic), 2001.

\bibitem{MR1225431}
Alexander~G. Reznikov.
\newblock Symplectic twistor spaces.
\newblock {\em Ann. Global Anal. Geom.}, 11(2):109--118, 1993.

\bibitem{MR2000f:53116}
Alexander~G. Reznikov.
\newblock Characteristic classes in symplectic topology.
\newblock {\em Selecta Math. (N.S.)}, 3(4):601--642, 1997.
\newblock Appendix D by Ludmil Katzarkov.

\end{thebibliography}
\bibliographystyle{plain}

\end{document}